\documentclass[twoside,11pt,amscd,reqno]{amsart}

\usepackage[bookmarks,
bookmarksnumbered,%
 colorlinks=true,%
 linkcolor=red,%
 citecolor=blue,%
 filecolor=blue,%
 urlcolor=blue,%
]
{hyperref}
\usepackage[T1]{fontenc}
\usepackage[latin1]{inputenc}
\usepackage{times}
\usepackage{euscript}
\usepackage{amssymb, amsmath, latexsym}
\usepackage{epigraph}
\usepackage{xypic}
\usepackage{appendix}
\usepackage{array}

\usepackage[OT2,OT1]{fontenc}

\newcommand\cyr
{
\renewcommand\rmdefault{wncyr}
\renewcommand\sfdefault{wncyss}
\renewcommand\encodingdefault{OT2}
\normalfont
\selectfont
}
\DeclareTextFontCommand{\textcyr}{\cyr}
\def\cprime{\char"7E }

\DeclareMathAlphabet{\curly}{OT1}{rsfs}{n}{it}
 
\renewcommand\appendix{\par
\setcounter{section}{0}%
\setcounter{subsection}{0}%
\setcounter{table}{0}
\setcounter{figure}{0}
\setcounter{equation}{0}
\gdef\thetable{\Alph{table}}
\gdef\thefigure{\Alph{figure}}
\gdef\theequation{\Alph{section}-\arabic{equation}}
\section*{Appendix}
\gdef\thesection{\Alph{section}}
\setcounter{section}{0}}

\def\CC{\mathbb{C}}
\def\PP{\mathbb{P}}

\def\ZZ{\mathbb{Z}}
  
\def\RR{\mathbb{R}}

\def\s-{\setminus}

\newtheorem{thm}{Theorem}[section]
\newtheorem{prop}[thm]{Proposition}
\newtheorem{lemma}[thm]{Lemma}

\newtheorem{cor}[thm]{Corollary}

\newtheorem{rmk}[thm]{Remark}

\numberwithin{equation}{section}

\begin{document}

\title[Real rational curves on K3 surfaces]
{Qualitative aspects of counting real rational curves on real K3 surfaces}

\author[Kharlamov]{Viatcheslav Kharlamov}

\address{
        IRMA UMR 7501, Universit\'e de Strasbourg, 7 rue Ren\'e-Descartes, 67084 Strasbourg Cedex,  FRANCE}

\email{kharlam@math.unistra.fr}

\author[R\u asdeaconu]{Rare\c s R\u asdeaconu}

\address{        
        Department of Mathematics, 1326 Stevenson Center, Vanderbilt University, Nashville, TN, 37240, USA}

\email{rares.rasdeaconu@vanderbilt.edu}



\date{\today}

\maketitle{}

\begin{abstract}
We study qualitative aspects of the Welschinger-like $\mathbb Z$-valued 
count of real rational curves on primitively polarized real $K3$ surfaces.
In particular, we prove that with respect to the degree of the polarization,
at logarithmic scale, the rate of growth of the number of such real rational curves  
is, up to a constant factor, the rate of growth of the number of complex rational 
curves.  We indicate a few instances when the lower bound for the number of 
real rational curves provided by our count is sharp. In addition, we exhibit various 
congruences between real and complex counts.
\end{abstract}

\maketitle

\thispagestyle{empty}

\epigraph{\scriptsize\cyr{\it{"V chislakh est\cprime  \, nechto, chego v slovakh, 
dazhe kriknuv ikh, net."}}
}
{\cyr{\it{Iosif Brodski{\u i}, 
Polden\cprime \,  v komnate.}}
}
\footnote{"In the numbers there is something, which in the words, even having 
shouted them, is not." Joseph Brodsky,
"Noon in a Room".}

\section*{Introduction}
The discovery by J.-Y.~Welschinger \cite{w} of a deformation 
invariant $\mathbb Z$-valued count of real rational curves
interpolating real collections of points on a real rational surface 
has allowed to respond in an affirmative way to the long standing 
problem of existence of real solutions in this enumerative 
problem. Moreover, the lower bound on the number of real 
solutions provided by the Welschinger invariants has happened
be so powerful that it allowed \cite{iks} to disclose a remarkable 
new {\it phenomenon of abundance}: in the logarithmic scale, 
when the degree of curves is growing,  the number of real solutions 
happens to be of about the same growth rate as the number of 
complex ones. Later on, similar abundance phenomena were 
observed in a few other, even more classical, enumerative problems, 
like enumerating linear subspaces on projective hypersurfaces  
(see \cite{kf2} and references therein). All this originates further 
natural questions, which are essential for applications:
{\it what are the asymptotic and arithmetical properties of the 
lower bounds provided by such an invariant $\mathbb Z$-valued count; 
how non trivial and sharp are these lower bounds?}

A response to  these questions requires a comparison with 
the behavior of the numbers of solutions over the complex field, 
which in the above mentioned problems are given by some 
Gromov-Witten and Schubert numbers, respectively. Up to our 
knowledge, the corresponding aspects of the complex 
enumerative algebraic geometry are rarely treated in the 
literature: for related information we refer the interested 
reader to \cite{IF}, \cite{GM}, and \cite{iks3}. The sharpness of 
Welschinger lower bounds is also little studied; here, we can cite only
\cite{w_o}.

In our previous paper \cite{kr}, we considered the problem of counting
real rational curves on primitively polarized real $K3$ surfaces, 
introduced an appropriate invariant $\mathbb Z$-valued count and 
expressed the answer in a closed form, which can be viewed as a 
real version of the Yau-Zaslow formula (see Sect. \ref{rem} below). 
Our aim here is to show that thus obtained invariant lower bounds have 
similar peculiar asymptotic and arithmetic properties as those that were 
observed in the previously studied real enumerative problems,
and to indicate some instances where the lower bounds are sharp.

The note is organized as follows. First of all we recall the precise statement 
of the real version of the Yau-Zaslow formula. Then, we start the qualitative 
analysis by relating our formula with the Dedekind eta-function and use one of 
Jacobi identities to establish some positivity property of the Welchinger invariants. 
In the next subsection we apply Hardy-Ramanujan-Uspensky results \cite{hr, Us} 
on the asymptotic behavior of the number of partitions to determine the asymptotic 
behavior of the Welschinger invariants in the logarithmic scale and to exhibit
an abundance phenomenon. The third subsection is devoted to the comparison 
modulo $2, 3, 4$, and $8$ of our Welschinger-type invariants, with the corresponding 
reduced Gromov-Witten invariants in the complex case, computed by the Yau-Zaslow
formula. Finally, in the last subsection we apply Kulikov's type I and II degenerations  
\cite{ku} to establish the sharpness of our lower bounds for certain real deformation 
types of real $K3$ surfaces. A closing section contains a couple of concluding remarks.
Some numerical data collected to illustrate the results obtained is shown in the table 
in the appendix.

\section{Real version of the Yau-Zaslow formula}
\label{rem}

Let $X$ be a generic real $K3$ surface admitting a complete 
real $g$-dimensional linear system of curves of genus $g.$ 
If $g\ge 2,$ assume, in addition, that $X$ is of Picard rank $1$ 
and the curves in the linear system belong to a primitive divisor class. 
Let $c_g$ denote the number of complex rational curves in this 
linear system, and $w_g=n_+-n_-$ the number of real rational curves 
in the same linear system counted with Welschinger sign; that is with 
the sign $+$, if the number of real solitary points is even, and with the sign 
$-$, otherwise (recall that by Chen's theorem \cite[Theorem 1.1]{chen}
the genericity assumption ensures that all the rational curves in our linear 
system are nodal).

The numbers $c_g$ depend only on $g$ and not on a specific choice 
of the surface $X$, and obey the Yau-Zaslow formula \cite{yz}

\begin{equation}
\label{yz-formula}
\sum_{g\geq 0} c_g q^g=
	\prod_{s\geq 1}\frac {1}{(1-q^{s})^{e_\CC}},
\end{equation}
where $e_\CC=24$ is the Euler characteristic of $X.$ As we proved in 
\cite{kr}, $w_g$ depends only on the Euler characteristic $e_\RR$ of
the real part $X_\RR$ of $X$, and for $e_\RR$ fixed the generating 
function for $w_g$ is as follows:
\begin{equation}
\label{ryz-formula}
\sum_{g\geq 0} w_g q^g=
	\prod_{r\geq 1}\frac {1}{(1+q^r)^{e_\RR}}
	\prod_{s\geq 1}\frac {1}{(1-q^{2s})^{\frac{e_\CC-e_\RR}{2}}}.
\end{equation}

Note that $e_\RR$ is always even and its values vary between $-18$ and 
$20$ (for references and more details on the topology of real $K3$ surfaces, 
see \cite{DK}).

\section{Analysis of the real version}

\subsection{Positivity}
\begin{thm}\label{positive-monotone} 
For each fixed $e_\RR< 0$, all $w_g$ are positive and form a 
strictly increasing sequence $|e_\RR| = w_1<w_2<w_3<\dots .$
For $e_\RR=0$, all $w_g$ with odd $g$ are zero, while those with 
even $g$ are positive and form a strictly increasing sequence 
$12=w_2<w_4<\dots $. For each fixed $e_\RR>0,$ all the terms 
of the sequence $(-1)^gw_g$ are positive and form a strictly increasing 
sequence $e_\RR = -w_1<w_2<-w_3<\dots .$
\end{thm}

\begin{proof} To prove these statements in the case $e_\RR\le 0$ 
it is sufficient to notice that the second product power series in 
(\ref{ryz-formula}) has nonzero coefficients only in even degrees 
and each of these coefficients is positive, while the first product 
power series has all coefficients positive as soon as $e_\RR<0.$
The strict monotonicity claim follows immediately by noticing that  
multiplying positive power series with increasing sequence of 
coefficients yields a positive series with increasing sequence of 
coefficients.

Assume now that $e_\RR>0.$ First, we rewrite our formula 
(\ref{ryz-formula}) in terms of the Dedekind eta-function 
$$
\eta(z)=e^{\pi i z/12}\prod_{n\ge 1} (1-e^{2\pi i nz})=
q^{\frac1{24}}\prod_{n\ge 1}(1-q^n),
$$ 
and of the modular discriminant
$$
\Delta(z)=\eta^{24}(z)=q\prod_{n\ge 1} (1-q^n)^{24},
$$
where $z$ is in the upper half-plane, and  $q=e^{2\pi iz}.$ 
In this notation, formula (\ref{ryz-formula}) can be written as
\begin{equation}
\label{eta-quotient}
\sum_{g\geq 0} w_g q^g=
\frac{q}{\sqrt{\Delta(2z)}}\left(\frac{\eta^2(z)}{\eta(2z)}\right)^{\frac{e_{\RR}}2}
\end{equation}

For the above eta-quotient there is the following remarkable 
Gauss identity (see, for example, \cite[Corollary 1.3]{koh}):
\begin{equation}
\label{theta}
\frac{\eta^2(z)}{\eta(2z)}=1+2 \sum_{n\ge 1}(-1)^n q^{n^2}.
\end{equation}
To finish the proof it is sufficient now to notice that the property 
to have nonpositive coefficients in odd degrees and nonnegative 
coefficients in even degrees is preserved under multiplication of 
power series (with such a property) and that the power series
\begin{equation}
\label{Delta}
\frac{q}{\sqrt{\Delta(2z)}}= \prod_{n\ge 1} (1-q^{2n})^{-12}
\end{equation}
has positive coefficients in each even degree and zero coefficients 
in each odd degree. Another possible approach is to replace 
$q$ by $-q$, which makes all the power series involved to have 
positive coefficients, and then to apply the same arguments as 
above in the case $e_\RR<0.$ This argument yields the strict 
monotonicity claim.
\end{proof}

\subsection{Asymptopia}

\begin{lemma}
\label{multiplication} 
Let  $\sum a_nq^n$ and $\sum b_nq^n$ be two power series with 
positive coefficients, and $\sum p_nq^n=(\sum a_nq^n)(\sum b_nq^n)$  
the product power series. If, at a logarithmic scale, the coefficients 
$a_n$ and $b_n$ have the asymptotic behavior 
$\log a_n\sim (an)^\alpha$ and $\log b_n\sim (bn)^\alpha,$  
for some real constant $0<\alpha <1,$ then 
$\log p_n\sim (cn)^\alpha$ where $c=(a^\frac{\alpha}{1-\alpha}
+ b^\frac{\alpha}{1-\alpha})^\frac{1-\alpha}{\alpha}.$
\end{lemma}

\begin{proof} The result follows from
$$
\max_{0\le k\le n} \log a_kb_{n-k}\le \log p_n\le \log n + 
\max_{0\le k\le n} \log a_kb_{n-k}
$$
and
$$
\max_{0\le k\le n} \log a_kb_{n-k}\sim \left((a^\frac{\alpha}{1-\alpha}
+ b^\frac{\alpha}{1-\alpha})^\frac{1-\alpha}{\alpha} n\right)^\alpha.
$$
To derive the latter relation it is sufficient to bound, from above and 
from below, the sequences $a_n$, $b_n$ by sequences of the type
$\exp(k_\pm + (1\pm\epsilon)(an)^{\alpha})$ and 
$\exp(k_\pm+(1\pm\epsilon)(bn)^\alpha),$ respectively, and then let 
$\epsilon$ go to $0.$ The bounds needed here can be found by using 
the fact that
$((a^\frac{\alpha}{1-\alpha}
+ b^\frac{\alpha}{1-\alpha})^\frac{1-\alpha}{\alpha} n)^\alpha$ is exactly 
the maximal value of the function $(at)^\alpha +(b(n-t))^\alpha$
on the interval $[0,n].$
\end{proof}

\begin{thm}\label{growth}
The following asymptotic relations hold:
\begin{itemize}
\item[ i)] If $e_\RR< 0$, then 
$$
\log w_{n}\sim \pi\sqrt{\frac{4(e_\CC- 3e_\RR)}{e_\CC}\cdot n}\sim 
\sqrt{\frac{e_\CC- 3e_\RR}{4e_\CC}} \log c_n.
$$
\item[ ii)] If $e_\RR >  0$, then
$$
\log\vert  w_{n} \vert \sim 2\pi\sqrt{n} \sim \frac12 \log c_n.
$$
\item[ iii)] If $e_\RR=0$, then 
$$ 
\log w_{2n}\sim \frac12 \log c_{2n}.
$$
\end{itemize}
\end{thm}

\begin{proof}
Hardy-Ramanujan and Uspensky results \cite{hr, Us}
on the the asymptotic behavior of the coefficients
in the power series expansions 
$$
\prod_{n\ge 1}(1+q^n)=\prod_{n\ge 1}\frac1{1-q^{2n-1}}= 
\sum_{n\ge 1} Q(n)q^n \quad\text{and}\quad
\prod_{n\ge 1}\frac1{1-q^{n}}= \sum_{n\ge 1}P(n)q^n
$$ 
give us the following equivalence relations:
\begin{equation}
\label{euler}
Q(n)\sim \frac{e^{\pi \sqrt\frac{n}3}}{4\cdot 3^{\frac14}\cdot n^\frac34}
\end{equation}
and 
\begin{equation}
\label{hardy}
P(n)\sim \frac{e^{\pi \sqrt\frac{2n}3}}{4n\sqrt{3}}.
\end{equation}
Thus, in the logarithmic scale, $\log P(n)\sim \log Q(2n)$ and 
$\log Q(n)\sim\pi\sqrt{n/3}.$ Then Lemma \ref{multiplication} 
implies that
$\log c_n\sim \pi \sqrt{e_\CC\cdot\frac{2n}3}=4\pi\sqrt{n}.$

If $e_\RR< 0$, then the coefficients $u_r$  in the power expansion 
$\sum u_rq^r$ of the first product in the  formula (\ref{ryz-formula}) 
are positive, and, according to Lemma \ref{multiplication} and formula 
(\ref{euler}), they grow in the log-scale as $\pi \sqrt{\frac{-e_\RR r}3}.$ 
The coefficients $v_{r}$ in the power expansion $\sum v_{r}q^{r}$ of the 
second product are vanishing in odd degree and they are positive in even 
degree. Lemma \ref{multiplication} and formula (\ref{hardy}) imply that 
the coefficients $v_{2s}$ grow in the log-scale as 
$\pi \sqrt{\frac{e_\CC-e_\RR}2 \cdot \frac{2s}3}.$ 
It follows that 
$$
\log w_{2n}\sim \pi \sqrt{\frac{e_\CC-3e_\RR}2 \cdot \frac{2n}3}.
$$
Wherefrom, by monotonicity of the sequence $w_n$,
$$
\log w_n\sim \pi \sqrt{\frac{e_\CC-3e_\RR}2 \cdot \frac{n}3}\sim 
\sqrt{\frac{e_\CC-3e_\RR}{4e_\CC}}\log c_n.
$$

If $e_\RR\ge 0$, then the proof is similar. This time we start from formula 
(\ref{eta-quotient}). Notice that according to formula (\ref{Delta}) and 
Lemma \ref{multiplication}, the coefficients of the first factor in the odd 
degrees are zero, while the coefficients in degree $2n$ grow in the log-scale 
as $\pi \sqrt{12\cdot \frac{2n}3} \sim \frac12 \log c_{2n}.$ Thus, there remains 
to notice that, due to formula (\ref{theta}), the coefficients of the second factor 
have polynomial growth, and, then, if $e_\RR>0$, to apply the monotonicity.
\end{proof}

\begin{cor} 
The number $r_g(X)$ of real rational curves  in the divisor class of the 
primitive polarization of $X$ satisfies the following bounds:
$$
\phi(g)=\vert w_g\vert \le r_g(X)\le c_g= \psi(g), 
$$
where  
$$
\log \psi(g)=
4\pi\sqrt{g}+o(\sqrt{g}),
$$
and, for a fixed $e_\RR\ne 0$,
$$
\log \phi(g)= 4\pi \rho\sqrt{g}+o(\sqrt{g})
$$
with $\rho=\frac12$ if $e_\RR> 0$ and 
$\rho=\sqrt{\frac{e_\CC- 3e_\RR}{4e_\CC}} $
if  $e_\RR<0$, while for $e_\RR=0$ it holds
$$
\log \phi(2g)= 2 \pi\sqrt{2g}+o(\sqrt{g}).
$$ 
\end{cor}

Similar abundance of real solutions phenomena are observed 
in several other real enumerative problems, see \cite{iks}, 
\cite{iks2}, \cite{BM}, \cite{kf1}, \cite{kf2}. There, like for $e_\RR\ge 0$ 
in the above Corollary, a magic factor $1/2$ occurs in quite a few
cases.

\subsection{Congruences.}
\label{congr}

The modularity of the generating series for the real and complex  
counting described in (\ref{yz-formula}) and (\ref{ryz-formula}) allows 
us to exhibit noteworthy congruences that go much further than 
$w_g\equiv c_g \pmod 2$, straightforward from definitions.

\begin{thm}
\label{congruence} 
The following congruences hold:

\begin{itemize}

\item[ i)] We have 
$w_g\equiv c_g \pmod2$ for any 
$g\geq 1,$ and $w_g\equiv c_g\equiv0\pmod2$  
for every $g$ with $g \not\equiv 0 \pmod8.$

\item[ ii)] If $e_\RR\equiv 0 \pmod 4$ 
then $w_g\equiv c_g \pmod4$ for any  $g\geq 1$, 
and if in addition $g\not\equiv 0\pmod 4$ then 
$w_g\equiv c_g\equiv0\pmod4$.

\item[ iii)] If $e_\RR\equiv 0 \pmod 8$ 
then $w_g\equiv c_g \pmod8$ for any  $g\geq 1$, 
and if in addition $g\not\equiv 0\pmod 2$ then 
$w_g\equiv c_g\equiv0\pmod8$.

\item[ iv)] If $e_\RR\equiv 0 \pmod 3$ and $g\not\equiv 0 \pmod3$ 
then $w_g\equiv c_g\equiv 0\pmod 3.$ 

\end{itemize}
\end{thm}

\begin{proof}
As a consequence of (\ref{theta}) and (\ref{Delta}) we get 
\begin{equation}
\label{parity}
\sum_{g\geq 0} w_g q^g\equiv \prod_{n\ge 1}\frac1{(1-q^{2n})^{12}} 
\equiv \prod_{n\ge 1}\frac1{(1-q^{8n})^{3}} \pmod 2.
\end{equation}
In particular, it implies that $w_g$ is even for every $g$ 
with $g \not\equiv 0 \pmod8$.

Likewise,
\begin{equation}
\label{parity-complex}
\sum_{g\geq 0} c_g q^g=\prod_{n\ge 1}\frac1{(1-q^{n})^{24}}
\equiv \prod_{n\ge 1}\frac1{(1-q^{8n})^{3}} 
\equiv \sum_{g\geq 0} w_g q^g\pmod 2.
\end{equation}

Making  an additional assumption that $e_\RR\equiv0\pmod 4,$ 
we get in a similar way
\begin{align*}
\sum_{g\geq 0} w_g q^g\equiv &~ \left(1+e_\RR\left(
\sum_{n=1}^{\infty} (-1)^nq^{n^2}\right)\right)
\prod_{n\ge 1}\frac1{(1-q^{2n})^{12}} \pmod4\\ 
\equiv &~ \prod_{n\ge 1}\frac1{(1-q^{2n})^{12}} \pmod4 \\
\equiv &~\sum_{g\geq 0} c_g q^g \pmod4.
\end{align*}
We conclude that $w_g$ and $c_g$ are congruent modulo 
$4$ for any $g$, and that they are divisible by $4$ for each 
$g\geq 1$ with $g\not\equiv 0 \pmod4,$  as it follows from 
$$
 (1-q^{2n})^{12}\equiv (1+ 2q^{4n}+q^{8n})^{3} \pmod4.
 $$

Assuming now that  $e_\RR\equiv0\pmod 8,$ the same proof 
as in the previous case yields 
$$
\sum_{g\geq 0} w_g q^g\equiv
\prod_{n\ge 1}\frac1{(1-q^{2n})^{12}}
\equiv \sum_{g\geq 0} c_g q^g \pmod8,
$$
and so we conclude that $w_g\equiv c_g\pmod 8$ for any 
$g\geq 1.$ Moreover, since 
 $$
 (1-q^{2n})^{12}\equiv (1+ 4q^{2n}+6q^{4n}+4q^{6n}+q^{8n})^{3} \pmod8,
 $$
we notice that $w_g\equiv c_g\equiv0\pmod8$ for any $g\equiv 1\pmod 2.$

If $e_\RR\equiv0\pmod 3$, we get in a similar way 
$$
\sum_{g\geq 0} w_g q^g\equiv 
\left(1+\sum_{n=1}^{\infty} (-1)^{n+1}q^{3n^2}
\right)^{\frac{e_\RR}{6}}\prod_{n\ge 1}\frac1{(1-q^{6n})^{4}} \pmod3
$$
and
$$
\sum_{g\geq 0} c_g q^g\equiv 
\prod_{n\ge 1}\frac1{(1-q^{3n})^{8}} \pmod3.
$$
We conclude that $w_g\equiv c_g\equiv 0\pmod 3$  
for each $g\geq 1$ with $g\not\equiv 0 \pmod3$.
\end{proof}

\subsection{Sharpness}

In \cite[Sect 5.2]{kr} we observed that the lower bound for the count 
of real rational curves on primitively polarized $K3$ surfaces given 
by the absolute value of $w_g$ is sharp for any $g$ as soon as the 
surface has no real points. The reasoning is simple: there is no real 
nodal rational curve in the corresponding linear system when $g$ is odd, 
since such a curve would have at least one real point among its singular  
points, which is impossible; while when $g$ is even, the singular points of 
such a curve split into pairs of conjugate ones, and therefore the curve 
counts with positive Welschinger sign. Here we prove that the lower bound 
is optimal in a few more cases.

\begin{thm}
\label{sharp}
For any $g$, the lower bound on the number of real solutions by the 
absolute value of $w_g$ is sharp for surfaces whose real locus is a 
torus or a pair of tori. 
\end{thm}

\begin{proof}
First, we treat the case of a pair of tori and $g$ odd. Let 
$\pi: Y\to \PP^1\times \PP^1$ be the double covering ramified 
along a real nonsingular curve of bi-degree $(4,4)$ without 
real points ({\it cf.,} \cite{vi,ns}), where $\PP^1\times \PP^1$ 
is considered with the standard product real structure (i.e., the 
hyperboloid). Such a double covering carries two real structures 
that differ by the automorphism of the covering, and we pick the 
one for which the real locus is formed by two tori. We denote by 
$F_1$ and  $F_2$ the generators of $\PP^1\times \PP^1.$ The 
linear system  $\vert F_1+n F_2\vert $ embeds $\PP^1\times \PP^1$ 
into $\PP^{2n+1},$ while its pull-back $\vert \pi^*(F_1+n F_2)\vert$ 
provides a representation of $Y$ as a hyperelliptic $K3$ surface 
in $\PP^{2n+1}.$ In such a representation the pull-back of hyperplane 
sections form a complete $2n+1$ dimensional linear system of curves 
of genus $g=2n+1.$ Finally, we take as $X$ an embedding into 
$\PP^{2n+1}$ of a real $K3$ surface obtained by a small generic 
real perturbation of $Y$ (its existence follows from the period 
space description, see for example \cite{ns}). Notice now that 
$Y_\RR$ consists of a pair of disjoint tori which are non-contractible 
in $\PP^{2n+1}_\mathbb R$ since the real locus of the starting 
ruled surface is non-contractible in $\PP^{2n+1}_\mathbb R.$
This implies that $X_\RR$ consists of a pair of disjoint, non-contractible 
tori as well. Hence, every real hyperplane section of $X_\RR$  has at 
least $2$ components, and thus can not be rational. This proves the 
sharpness claim, since according to formula (\ref{ryz-formula})
we have $w_g=0$ for odd values of $g,$ as soon as $e_\mathbb R$ is zero.

\smallskip

To treat the case of a pair of tori and $g$ even, we consider a 
degeneration of a $K3$ to a double covering of the blown up 
projective plane. Namely, we start from $\PP^2(1)$, the projective 
plane blown up at a real point, and consider a double covering 
$\varpi: Y\to \PP^2(1)$ ramified along the proper transform 
of a one-nodal sextic with a real solitary point at the center 
of the blow-up and two ovals surrounding this solitary point. 
The standard real structure on $\PP^2(1)$ lifts to two real 
structures on $Y$, and we pick the one for which the real locus 
consists of two tori. We embed now $\PP^2(1)$ into $\PP^{2n}$ 
by the linear system $\vert E+n F\vert $, where $E$ is the exceptional 
divisor and $F$ stands for the straight lines through the center of the 
blow-up. As above, we take $X\subset \PP^{2n}$  to be a  generic small 
real perturbation of $Y.$ The $K3$ surface $Y$ carries a natural real 
elliptic fibration given by the pull-back of the pencil of lines through the 
center of the blow-up. Since the starting sextic admits no real tangents 
passing through the node, all the singular fibers of this elliptic fibration 
are imaginary. This implies that every real rational hyperplane section 
of $X$, which as is known ({\it cf.,} \cite[Proposition 4.1]{BL}) is a 
perturbation of the section, $ \varpi^{-1}(E)$, and a  collection of, 
possibly multiple, singular fibers, has no real singular points. Therefore, 
all the inputs into $w_g$ in such a construction are positive, wherefrom 
the sharpness for this other particular case: $X_\RR$ is a pair of tori 
and $g$ is even, $g=2n.$ Notice that in this case $w_g>0$ and it grows 
fast, as discussed in Theorems \ref{positive-monotone} and 
\ref{growth}, respectively.

\smallskip

In the case of one torus and $g$ even, we start again from $\PP^2(1),$ 
the projective plane blown up at a real point, and consider a double covering 
$\varpi: Y\to \PP^2(1)$ ramified in a proper transform of a one-nodal sextic 
with a real solitary point at the center of the blow-up and, this time, no other 
real points. As a real structure on $Y$, we select that lift of the standard real  
structure on $\PP^2(1)$ for which the real locus of $Y$ is a torus. We embed 
$\PP^2(1)$ into $\PP^{2n}$ by the linear system $\vert E+n F\vert $, where 
$E$ is the exceptional divisor and $F$ stands for the linear system of lines 
through the center of the blow-up. As above, we can assume that all the 
singular fibers of the associated elliptic fibration are imaginary. For that it is 
sufficient to make our sextic generic, since a generic nodal sextic has no double 
tangents passing through the node. Hence, we can argue as in the previous case, 
that is to use the pull-back $\vert \varpi^*(E+n F)\vert $ to represent the double 
covering under consideration by a hyperelliptic $K3$ surface in $\PP^{2n},$ and 
then to take as $X$ a generic small real perturbation of $Y$. Once more all the 
inputs in such a construction are positive.  Wherefrom the sharpness for this 
other particular case:  $X_\RR$ is a torus and $g$ is even, $g=2n.$ Notice again 
that in this case again $w_g>0$ and it grows fast.

\smallskip

To construct  sharp examples for the remaining case, a $K3$ surface 
in $\PP^{2n+1}$ with a torus as the real locus, we start from a suitable 
type II degeneration of such a $K3$ surface \cite{ku}. Namely, we start, 
as in the case of a pair of tori,  from the same real rational ruled surface 
$\PP^1\times \PP^1=\Sigma_0\subset \PP^{2n+1}$ and consider an elliptic 
normal curve $E$  without real points that is cut on this ruled surface 
by a real rational ruled surface $\Sigma_2$ having empty real locus and 
intersecting $\Sigma_0$ transversally along $E$ with $E^2=-2$ on this 
$\Sigma_2.$ To construct an explicit example, {\it cf.,} 
\cite[Lemma 1, page 643]{ci-mi}, one can start from the elliptic curve 
$\mathbb C^2/(\mathbb Z+i\mathbb Z)$ equipped with the real structure 
$z\mapsto \bar z + \frac12$, embed it into $\PP^{2n+1}$ by means of the 
linear system $n(A+B)+(\frac12+A)+(\frac12+B)\sim (n+1)(A+B)$
taking $A=0, B=\frac{i}2$, and choose as a scroll $\Sigma_0=\PP^1\times \PP^1$
containing $E,$ the scroll corresponding to the hyperelliptic involution defined 
by the divisor $A+C$, $C=\frac12.$ Under such choices the scroll $\Sigma_0$ 
becomes real and has a torus as its real part, while the scroll corresponding 
to the divisor $A+B$ gives us a real $\Sigma_2$ with an empty real part, 
as required. After that, there remains to pick as $X$ a small generic real 
perturbation of $\Sigma_0\cup\Sigma_2.$ The existence of a smooth such 
surface $X$ is guaranteed by \cite[ Theorem 1, page 644]{ci-mi}. The generic 
such surface $X$ is primitively polarized \cite[Theorem 2, page 646]{ci-mi}, 
and, as always due to the surjectivity of the period map, it can be deformed 
further to a surface generic in the sense of Chen's theorem \cite{chen}, 
which guarantees that the rational curves in the primitive polarization are 
all nodal.

Under such a choice, $X$ does not have  any real rational hyperplane section. 
Indeed, if such a section exists, then by the compactness of Kontsevich's space 
of stable curves, there would exist a real projective connected nodal curve $Z$ 
of arithmetic genus $0$ and a real regular map $f: Z\to \Sigma_0\cup\Sigma_2$ 
that realizes a hyperplane section of $\Sigma_0\cup\Sigma_2.$ On the other hand, 
the hyperplane sections of $\Sigma_0$ form the linear system $\vert F_1+n F_2\vert,$ 
where $F_1,F_2$ are generators of $\PP^1\times \PP^1=\Sigma_0,$ while the 
hyperplane sections of $\Sigma_2$ form the linear system $\vert F_1'+(n+1)F_2'|,$ 
where $F_1'$ stands for  the $(-2)$-section and $F_2'$ for the generator of $\Sigma_2.$ 
The standard generators $F_1, F_2$ and the $(-2)$-section are complex conjugation 
invariant. Hence, every real rational map $f: Z\to \Sigma_0\cup\Sigma_2$ as above, 
which represents a hyperplane section of $\Sigma_0\cup\Sigma_2,$ should have in 
its source, $Z$, an irreducible real component of type $F_1+aF_2$ and an irreducible 
real component of type $F'_1+aF'_2$. Since $Z$ is rational and connected, each two 
real components should be connected by a chain of real components. This is impossible 
for real components belonging one to $\Sigma_0$ and another to $\Sigma_2$, since any 
two such components intersect only at pairs of complex conjugated points. Such a 
contradiction ends the proof.
\end{proof}

\section{Concluding remarks}

\subsection{On asymptotics}

Hardy and Ramanujan have obtained not only an asymptotic approximation
for coefficients $P(n), Q(n)$ (see formulas (\ref{euler}), (\ref{hardy})), 
but also a full asymptotic expansion, which later was even improved by 
Rademacher up to a convergent series expression. These can be applied to 
get similar expansions for $w_g.$ We have restricted ourselves here to 
asymptotic approximations, since for our aim, to demonstrate the abundance 
phenomenon, it is not necessary to go further. Moreover, it would only 
obscure the presentation by a much heavier and lengthy analysis.

\subsection{On congruences modulo $2$}

Table \ref{defaulttable} below that gives the values of  $w_g$ and $c_g$ 
for $g\le 20$ shows that these values are odd if $g=8$ and $16$.
Thus, it may give the impression that they should be odd for all $g\equiv0\pmod 8$ 
({\it cf.,} Section \ref{congr}). In fact, the situation is much different.

Let $\{i_n\}_{n\geq 0}$ be the parity sequence given by 
$i_n\equiv w_{8n}\equiv c_{8n}\pmod 2,\,n\geq 0.$

\begin{prop} 
The sequence $\{i_n\}_{n\geq 0}$ contains infinitely many 
$0$'s and infinitely many $1$'s.
\end{prop}

\begin{proof} As noticed in (\ref{parity}), the following identity holds 
$$
\sum_{n\geq 0}  i_nq^{8n}\equiv \sum_{g\geq 0}  w_gq^g 
\equiv \prod_{n\geq1}\frac1{(1-q^{8n})^{3}} \pmod2.
$$
Moreover, we have
$$
\prod_{n\geq1}\frac1{(1-q^{8n})^{3}}
\equiv \prod_{n\geq1}\frac{(1-q^{16n})^{2}}{(1-q^{8n})}
\prod_{n\geq1}\frac1{(1-q^{16n})^{3}} \pmod2.
$$
Since  
$$
\prod_{n\geq1}\frac1{(1-q^{16n})^{3}} \equiv \prod_{n\geq1}\frac1{(1-q^{n})^{48}} 
\equiv \left(\sum_{n\geq 0}  i_nq^{8n}\right)^2\pmod2,
$$
we obtain
\begin{equation}
\label{recursion}
\sum_{n\geq 0}  i_nq^{8n}\equiv \prod_{n\geq1}\frac{(1-q^{16n})^{2}}{(1-q^{8n})}
\left(\sum_{n\geq 0}  i_nq^{16n}\right)\pmod 2.
\end{equation}

Furthermore, from the Jacobi identity \cite[Corollary 1.4]{koh}
\begin{equation*}
\label{Jacobi}
\prod_{n\geq 1}\frac{(1-q^{16n})^{2}}{(1-q^{8n})} 
\equiv \prod_{n\geq 1} (1-q^{8n})^3
\equiv  \sum_{n=0}^\infty q^{4n(n+1)} \pmod 2
\end{equation*}
we notice that the power series development of the term 
$\prod_{n\geq 1}\frac{(1-q^{16n})^{2}}{(1-q^{8n})}$ in (\ref{recursion})
contains infinitely many odd and infinitely many even coefficients and the 
gaps between odd coefficients are growing quadratically.

Suppose  now that the sequence $\{i_n\}_{n\geq 0}$ contains finitely many $1$'s.
That would imply (in arithmetics over $\mathbb Z/2$) the equality between a 
non-zero polynomial and a product of an infinite series with a polynomial, which is 
impossible.  Furthermore, if  we suppose now that the sequence $\{i_n\}_{n\geq 0}$ 
contains finitely many $0$'s, to get a contradiction, it is sufficient to write 
$\sum_{n\geq 0}  i_nq^{8n}$ as a sum of a polynomial with $\sum_{n\ge C}q^{8n},$ 
and to observe that the coefficients of $\sum_{n\geq 0}  i_nq^{8n}$ in powers 
$4(4k+1)(4k+2)$ and $4(4k+2)(4k+3)$ have opposite parities for all $k$ sufficiently 
large with respect to $C,$ which is impossible.
\end{proof}

\subsection{Other congruences}

In Section \ref{congr}, we discussed congruences between 
the real and complex invariants, modulo primes and their powers 
that divide $24,$ i.e., congruences modulo $2, 3, 4,$ and $8.$ 
In fact, when the powers of $2$ and $3$ do not divide $24,$ 
interesting congruences are still expected to occur. We discuss 
below some congruences modulo $9$ and $16.$

\begin{prop}
The following congruences hold:
\begin{itemize}
\item[ i)] If $e_\RR\equiv 0\pmod 9$ then 
$w_g\equiv c_g\equiv 0 \pmod9$ for all $g\equiv 4\pmod 6.$
\item[ ii)] If $e_\RR\equiv 0\pmod{16}$ then 
$w_g\equiv c_g\equiv 0 \pmod{16}$ for all 
odd $g>1.$
\end{itemize}
\end{prop}

\begin{proof} 
We assume first that $e_\RR$ is divisible by $9$  
(it happens with $e_\RR=-18, 0,$ and $18$), and 
discuss the congruence of the two invariants modulo $9.$

As Guo-Niu Han kindly pointed to us, the Jacobi  identity 
\cite[Corollary 1.4]{koh} shows that
$$
\prod_{n\geq 1} {(1-q^n)}^3 =A(q^3) + 3 q B(q^3),
$$
where  $A, B \in \ZZ[[q]],\,A(0)=1,$ which implies that for any integer $k,$ 
\begin{equation}
\label{han-relation}
\prod_{n\geq 1} {(1-q^n)}^{3k}=A_k(q^3) + 3 q B_k(q^3) + 9 q^2 C_k(q^3),
\end{equation} 
for some power series $A_k, B_k, C_k\in \ZZ[[q]],\,A_k(0)=1.$ 
In particular, when $k=-8$, it makes evident the congruence
$c_g\equiv 0\pmod 9$ for all $g\equiv 2\pmod 3.$

To prove the claim for the real invariant, observe first that 
$$
\left(1+2\sum_{n\ge 1} (-1)^nq^{n^2}\right)^9\equiv E(q^3) \pmod 9,
$$
where $E\in \ZZ[[q]],\,E(0)=1.$ This implies that 
$$
\left(1+2\sum_{n\ge 1} (-1)^nq^{n^2}\right)^{9\ell}\equiv E_l(q^3)\pmod 9,
$$ 
for any $\ell\in\ZZ,$ where $E_\ell\in \ZZ[[q]],\,E_\ell(0)=1.$ 
Thus, using (\ref{theta}), (\ref{Delta}), and (\ref{han-relation}), we find that 
$$
\sum_{g\geq 0} w_g q^g\equiv E_{e_\RR/18}(q^3)
\left(A_{-4}(q^6)+3q^2B_{-4}(q^6)+9q^4C_{-4}(q^6)\right) \pmod9.
$$
The multiplication of $A_{-4}(q^6)+3q^2B_{-4}(q^6)$ by 
$E_{e_\RR/18}(q^3)$ does not produce any term of 
exponent $6k+4.$ Therefore, $w_g\equiv 0\pmod 9,$ 
for all $g$ such that $g\equiv 4\pmod6,$ and claim i) is proved.

Congruences modulo $16$ for the complex invariants $c_g$ can be detected by 
using Klein's modular function $j(z)$ and the classical formula
$$
j(q)=\frac1q+\sum_{n=0}^\infty a(n)q^n	
=\frac{(1+240 \sum_{n=1}^\infty \sigma_3(n)q^n)^3}{q \Pi_{n=1}^\infty(1-q^n)^{24}},
$$
where $\sigma_3(n)=\sum_{d|n}d^3.$ 
As
$(1+240 \sum_{n=1}^\infty \sigma_3(n)q^n)^3\equiv 1\pmod{16},$ we have

\begin{equation}
\label{complex-case}
\sum_{g=0} c_gq^g\equiv qj(q)\pmod{16}.
\end{equation}
On the other hand, according to Lehner \cite{le}, one has 
$a(2k)\equiv0\pmod{2^{11}},$ for all $k>0.$ Hence, $c_g$ is divisible by 
$16$ for all $g>1$ with $g\equiv 1\pmod 2.$

If $e_\RR\equiv0\pmod {16}$ (it happens with $e_\RR=0, 16$, and $-16$), then 
the same arguments as in the proof of congruences modulo $4$ and $8$ show that 
$$
w_g\equiv  \prod_{n\ge 1}\frac1{(1-q^{2n})^{12}}\pmod {16}.
$$ 
This implies $w_g\equiv0\pmod {16}$ for any odd $g>1,$ and claim ii) is proved.
\end{proof}

\begin{rmk}
The congruence (\ref{complex-case}) holds modulo $9$ as well, and 
$a(3k)\equiv 0\pmod{3^5}$ for all $k>0$ according to \cite{le}. Hence 
$c_g\equiv 0\pmod 9$ if $g>1$ and $g\equiv 1\pmod 3.$
\end{rmk}

\subsection{On sharpness}

A couple of other instances of real $K3$ surfaces where  the lower bound 
for the number of real rational curves given by $\vert w_g\vert$ is optimal 
were already pointed in our previous paper \cite{kr}. One such example was 
the case of Harnack surfaces of degree 4 in $\mathbb P^3$.

On the other hand,  for real nonsingular $K3$ surfaces of degree $4$ in 
$\PP^3$ having the real locus consisting of $6$ spheres and a sphere 
with $5$ handles the lower bound given by $\vert w_g\vert$, which is equal 
to $48$ in this special case,  is not sharp.

\begin{prop} The number of real rational hyperplane sections of a generic real 
$K3$ surface degree $4$ in $\PP^3$ is at least $272$ if the real locus of the 
surface consists of $6$ spheres and a sphere with $5$ handles.
\end{prop}

\begin{proof} 
The Euler characteristic $e_\RR$ of the real part of such a surface is $4,$ 
and thus for such surfaces  $w_g=-48$ (here $g=3$). But, in fact, there are 
always at least $272$ real rational hyperplane sections on such a surface. 
Indeed, each real line and each real plane intersect the component of genus 
$5,$ as such a component is not homotopy trivial in $\mathbb R\PP^3$ 
(see \cite[Theorem II]{kh}). By B\'ezout theorem, it implies that each of the 
$6$ spheres is convex (that is bound convex balls) and each $2$ of them are 
contained in a convex set disjoint from other spheres. By genus argument, 
a real hyperplane through $3$ of the $6$ spheres does not intersect the $3$ 
others, and thus these $3$ spheres, as any $3$ disjoint convex spheres in a 
real affine 3-space such that each $2$ of them are contained in a convex set 
disjoint from the third, have $8$ common supporting planes. Each of the 
supporting planes gives us a real rational curve with $3$ solitary points, 
hence of Welschinger weight $-1$. Since the total number of supporting 
planes obtained in such a way is equal to $8\times \binom{6}{3}=160$ and 
$w_g=-48$, the total number of real rational curves is at least 
$160 + (160-48)=272.$
\end{proof}

The same argument can be applied to other deformation classes 
of real nonsingular $K3$ surfaces of degree $4$ in $\PP^3$ whose 
real locus contains a non contractible component and $m\ge 3$ components 
contractible in $\mathbb R\PP^3$ (for a full deformation classification of 
real nonsingular $K3$ surfaces of degree $4$ in $\PP^3$ one can look at 
the survey \cite{DK} or at the original Nikulin's paper \cite{nik}; some of 
them have $e_\RR=0$ and are different from tori). More precisely, as above 
by B\'ezout theorem each of the contractible components is an affine convex 
sphere and each $2$ of them are contained in a convex set disjoint of the third, 
while, as it is easy to show, any $3$ affine convex spheres with such a property
are contained in a common affine space. Thus, as above each $3$ contractible 
components provide $8$ tritangent planes and give an input $-8$ into $w_g.$ 
In particular, for these kind of $K3$ surfaces, if $e_\RR=0$ we get 
$8\times \binom{m}3$ as an improved lower bound.

%%%%%%%%%%%%%%%%%%%%%%%%%%%%%%%%%%%%%%%%%%%%
%%%%%%%%%%%%%%%%%%%%%%%%%%%%%%%%%%%%%%%%%%%%

\appendix

The table below is based on formulas (\ref{ryz-formula})  and (\ref{yz-formula}). 
It provides the number of real rational curves counted with the Welschinger sign 
on primitively polarized $K3$ surfaces of degrees up to twenty, in the cases when 
$e_\RR=0,-18$, and $20.$ The last column gives, for comparison,  the corresponding 
number of complex curves.

\begin{table}[htdp]
\begin{centering}
\small{
\begin{tabular}{|r|r|r|r|r|} 
\hline &\multicolumn{3}{c|}{Real Case} &Complex Case\\
\cline{2-4}
$g$ & $e_\RR=0$ & $e_\RR=-18$  & $e_\RR=20$ &\\
\hline
0& 1 & 1 & 1 & 1 \\
1& 0 & 18 & -20 & 24 \\
2& 12 & 192 &192 & 324 \\
3& 0 & 1536 & -1200 & 3200 \\
4& 90 & 10152 & 5630 &25650\\
5& 0 & 58284 & -21744&176256\\
6& 520 & 299776 & 73600 &1073720\\
7& 0 & 1410048 & -226688 &  5930496\\
8& 2535 & 6155079 & 648195 & 30178575 \\
9& 0 &  25207736 & -1742320 & 143184000\\
10& 10908 & 97675200 & 4446912 & 639249300\\
11& 0 & 360471552 &-10863840 & 2705114880\\
12& 42614 & 1273876088 &25553402 & 10914317934\\
13& 0 & 4329852624 & -58129280 &4 2189811200\\
14& 153960 & 14207361792 & 128365440 & 156883829400\\
15& 0 & 45144664064 & -276044032 & 563116739584\\
16& 521235 & 139288329729 & 579574795 & 1956790259235\\
17& 0 & 418257062220 & -1190636016 & 6599620022400\\
18& 1669720 & 1224808431104 & 2397710720 & 21651325216200\\
19& 0 & 3503958594048 & -4740978480 &69228721526400\\
20& 5098938& 9808358121720 &9217285614  & 216108718571250\\
\hline
\end{tabular} 
}
\smallskip
\caption{Numbers of real rational curves vs. complex curves on $K3$ surfaces}
\end{centering}
\label{defaulttable}
\end{table}

\subsection*{Acknowledgements} Both authors would like to thank Vanderbilt 
University  for its hospitality while this work was finalized. The second author 
acknowledges the support of the Simons Foundation's "Collaboration Grant for 
Mathematicians", award number 281266. We thank Alex Degtyarev who provided 
the computer program exhibiting the numbers included in the appendix, and 
Guo-Niu Han for invaluable help with proving congruences modulo $9.$ 
We are grateful to the anonymous referee for pointing us several discords in 
the first version of the  note.

%%%%%%%%%%%%%%%%%%%%%%%%%%%%%%%%%%%%%%%%%%%%
%%%%%%%%%%%%%%%%%%%%%%%%%%%%%%%%%%%%%%%%%%%%

\bibliographystyle{alpha} 

\begin{thebibliography}{7}





\bibitem{BL}
{\sc J.  Bryan,~N.~C.~Leung,} 
{\em The enumerative geometry of K3 surfaces and modular forms.} 
J. Amer. Math. Soc. 13 (2000), no. 2, 371--410.




\bibitem{BM}
{\sc E.~Brugall\'e,~G.~Mikhalkin,}
{\em Enumeration of curves via floor diagrams.}
C. R. Math. Acad. Sci. Paris 345 (2007), no. 6, 329--334. 






\bibitem{chen}
{\sc X.~Chen,}
{\em A simple proof that rational curves on $K3$ are nodal.} 
Math. Ann. 324 (2002), no. 1, 71--104. 






\bibitem{ci-mi}
{\sc C.~Ciliberto,~A.~Lopez,~R.~Miranda,}
{\em Projective degenerations of $K3$ surfaces, Gaussian maps, and Fano threefolds.}
Invent. Math. 114 (1993), no. 3, 641--667.

 
 
 
\bibitem{DK}
{\sc A.~Degtyarev,~V.~Kharlamov,}
{\em Topological properties of real algebraic varieties : du c\^{o}t\'{e} de chez Rokhlin.}
Uspekhi Mat. Nauk. 55 (2000), no. 4, 129--212.




\bibitem{kf1}
{\sc S.~Finashin,~V.~Kharlamov,}
{\em Abundance of real lines on real projective hypersurfaces.}
Int. Math. Res. Notices vol. 2013, no. 16 (2013), 3639--3646.




\bibitem{kf2}
{\sc S.~Finashin,~V.~Kharlamov,}
{\em Abundance of $3$-planes on real projective hypersurfaces.}
 Arnold Math. J. 1 (2015), no. 2, 171--199.




\bibitem{IF}
 {\sc P.~Di Francesco,~C.~Itzykson,}
{\em Quantum intersection rings.}
In the book:  The moduli space of curves (Texel Island, 1994), 81--148, 
Progr. Math., 129, Birkh\"auser Boston, Boston, MA, 1995.




\bibitem{GM}
{\sc D.~B.~Gr\"unberg,~P.~Moree,}
{\em Sequences of enumerative geometry: congruences and asymptotics.}
With an appendix by Don Zagier. Experiment. Math. 17 (2008), no. 4, 409--426. 




\bibitem{hr}
{\sc G.~H.~Hardy,~S.~Ramanujan,}
{\em Asymptotic formulae in combinatory analysis.}
Proc. London Math. Soc. 17 (1918), 75 --115. 




\bibitem{iks}
{\sc I.~Itenberg,~V.~Kharlamov,~E.~Shustin,}
{\em Welschinger invariant and enumeration of real rational curves.}
 Int. Math. Res. Notices, vol. 2003, no. 49 (2003),  2639--2653.
 
 
 
 
\bibitem{iks2}
{\sc I.~Itenberg,~V.~Kharlamov,~E.~Shustin,}
{\em  Welschinger invariants of real del Pezzo surfaces of degree $\ge2$.}
Internat. J. Math. 26 (2015), no. 8, 1550060, 63 pp.




\bibitem{iks3}
{\sc I.~Itenberg,~V.~Kharlamov,~E.~Shustin,}
{\em  Logarithmic asymptotics of the genus zero Gromov-Witten invariants of the blown up plane.}
Geom. Topol. 9 (2005), 483--491. 




\bibitem{kh}
{\sc V.~Kharlamov,}
{\em Additional congruences for the Euler characteristic of even-dimensional real algebraic varieties.}
Funkcional. Anal. i Prilo\v{z}en. 9, no. 2 (1975), 51--60. 




\bibitem{kr}
{\sc V.~Kharlamov,~R.~R\u asdeaconu,}
{\em Counting real rational curves on $K3$ surfaces.}
Int. Math. Res. Notices, vol. 2015, no. 14 (2015), 5436--5455.



\bibitem{koh}
{\sc G.~K\"ohler,} 
{\em Eta Products and Theta Series Identities.} 
Springer-Verlag, 2011.




\bibitem{ku}
{\sc Vik.~S.~Kulikov,} 
{\em Degenerations of $K3$ surfaces and Enriques surfaces.}
Izv. Akad. Nauk SSSR Ser. Mat. 41 (1977), no. 5, 1008--1042.



\bibitem{le}
{\sc J. Lehner,}
{\em Further congruence properties of the Fourier coefficients of the modular invariant $j(\tau)$.}
American Journal of Mathematics 71 (1949), no. 2, 373--386.




\bibitem{nik}
{\sc V.~V.~Nikulin,} 
{\em Integer symmetric bilinear forms and some of their geometric applications.}
Math USSR-Izv. 14 (1980), no. 1, 103--167.





\bibitem{ns}
{\sc V.~V.~Nikulin,~S.~Saito,}
{\em Real $K3$ surfaces with non-symplectic involution and applications.}
Proc. London Math. Soc. (3) 90  (2005), no. 3, 591--654. 




\bibitem{Us}
{\sc J.~V.~Uspensky,}
{\em Asymptotic formulae for numerical functions which occur in the theory of partitions.}
Bull. Acad. Sci. URSS (6) 14 (1920), 199--218.




\bibitem{vi}
{\sc O.~Ya.~Viro,}
{\em Construction of multicomponent real algebraic surfaces.}
Dokl. Akad. Nauk SSSR 248 (1979), no. 2, 279--282.




\bibitem{w} 
{\sc J.-Y.~Welschinger,}
{\em Invariants of real rational symplectic $4$-manifolds and lower bounds in real enumerative geometry.}
C. R. Math. Acad. Sci. Paris  336 (2003), no. 4, 341--344.
 
 
 
 
\bibitem{w_o} 
{\sc J.-Y.~Welschinger,}
{\em Optimalit\'e, congruences et calculs d'invariants des vari\'et\'es symplectiques r\'eelles de 
dimension quatre.}
 arXiv:0707.4317.
    
    
    
 
\bibitem{yz}
{\sc S.-T.~Yau,~E.~Zaslow,}
{\em BPS states, string duality, and nodal curves on $K3.$} 
Nuclear Phys. B 471 (1996), no. 3, 503--512. 




\end{thebibliography}

\end{document}